\documentclass{article}
\usepackage{arxiv}

\usepackage[utf8]{inputenc} %
\usepackage[T1]{fontenc}    %
\usepackage{hyperref}       %
\usepackage{url}            %
\usepackage{booktabs}       %
\usepackage{amsfonts}       %
\usepackage{nicefrac}       %
\usepackage{microtype}      %
\usepackage{lipsum}

\usepackage{amsmath}
\usepackage{amssymb}
\usepackage{amsthm}
\usepackage[toc,page]{appendix}

\usepackage{lineno}

\newtheorem{theorem}{Theorem}
\newtheorem{lemma}{Lemma}
\newtheorem{definition}{Definition}
\newtheorem{corollary}{Corollary}

\DeclareMathOperator{\sign}{sgn}

\usepackage{mathptmx}      %
\begin{document}

\title{Approximation Rates for Neural Networks with General Activation Functions
}

\author{Jonathan W. Siegel \\
  Department of Mathematics\\
  Pennsylvania State University\\
  University Park, PA 16802 \\
  \texttt{jus1949@psu.edu} \\
  \And Jinchao Xu \\
  Department of Mathematics\\
  Pennsylvania State University\\
  University Park, PA 16802 \\
  \texttt{jxx1@psu.edu} \\
}

\maketitle

\begin{abstract}
We prove some new results concerning the approximation rate of neural networks with general activation functions.
Our first result concerns the rate
of approximation of a two layer neural network with a polynomially-decaying non-sigmoidal activation function. We extend the dimension independent approximation rates previously obtained to this new class of activation functions. Our second result gives a weaker, but still dimension independent, approximation rate for a larger class of activation functions, removing the polynomial decay assumption. This result applies to any bounded, integrable activation function. Finally, we show that a stratified sampling approach can be used to improve the approximation rate for polynomially decaying activation functions under mild additional assumptions.
\end{abstract}

\section{Introduction}
Deep neural networks have recently revolutionized a variety of areas of machine learning,
including computer vision and speech recognition \cite{lecun2015deep}. A deep neural network
with $l$ layers is a statistical model which takes the following form
\begin{equation}
 f(x;\theta) = h_{W_l,b_l} \circ \sigma \circ h_{W_{l-1},b_{l-1}}\circ \sigma \circ \cdots \circ \sigma \circ h_{W_1,b_1},
\end{equation}
where $h_{W_i,b_i}(x) = W_ix + b_i$ is an affine linear function, $\sigma$ is a fixed activation function
which is applied pointwise and $\theta = \{W_1,...,W_l,b_1,...,b_l\}$ are the parameters of the model.

The approximation properties of neural networks have recieved a lot of attention, with many positive results. For example,
in \cite{leshno1993multilayer,ellacott1994aspects} it is shown that neural networks can appproximate any function on a compact set as 
long as the activation function is not a polynomial, i.e. that the set 
\begin{equation}
   \label{ShallowNN}
	\Sigma_d(\sigma)=\mathrm{span}\left\{\sigma(\omega\cdot x+b):\omega\in\mathbb{R}^d,b\in\mathbb{R}\right\}
\end{equation}
is dense in $C(\Omega)$ for any compact $\Omega\subset \mathbb{R}^d$. An earlier result of this form can be found in
\cite{hornik1993some}, and \cite{hornik1991approximation} shows that derivatives can be approximated arbitrarily accurately as well.
An elementary and constructive proof for $C^\infty$ functions can be found in \cite{attali1997approximations}.

In addition, 
quantitative estimates on the order of approximation are obtained for sigmoidal activation functions in \cite{barron1993universal}
and for periodic activation functions in \cite{mhaskar1994dimension} and \cite{mhaskar1995degree}. Results for general
activation functions can be found in \cite{hornik1994degree}. A remarkable feature of these results in that the approximation rate is $O(n^{-1/2})$, where $n$ is the number of hidden neurons, which shows that neural networks can overcome the curse of dimensionality. Results 
concerning the approximation properties of generalized translation networks (a generalization of two-layer neural networks) 
for smooth and analytic functions are obtained in \cite{mhaskar1996neural}.
Approximation estimates for multilayer convolutional neural networks are considered in \cite{zhou2018universality}
and multilayer networks with rectified linear activation functions in \cite{yarotsky2017error}.
A comparison of the effect of depth vs width on the expressive power of neural networks is presented in \cite{lu2017expressive}. 

An optimal approximation rate in terms of highly smooth Sobolev norms is given in \cite{petrushev1998approximation}. This work differs from previous work and the current work in that it considers approximation of highly smooth functions, for which proof techniques based on the Hilbert space structure of Sobolev spaces can be used. In contrast, the line of reasoning initially persued in \cite{barron1993universal} and continued in this work makes significantly weaker assumptions on the function to be approximated.

A review of a variety of known results, especially for networks with one hidden layer, 
can be found in \cite{pinkus1999approximation}. More recently, these results have been improved by a factor of $n^{1/d}$ in \cite{klusowski2016uniform} using the idea of stratified sampling, based in part on the techniques in \cite{makovoz1996random}.  

Our work, like much of the previous work, focuses on the case of two-layer neural networks. A two layer neural network
can be written in the following particularly simple way
\begin{equation}
 f(x;\theta) = \displaystyle\sum_{i = 1}^n \beta_i\sigma(\omega_i \cdot x + b_i),
\end{equation}
where $\theta = \{\omega_1,...,\omega_n,b_1,...,b_n\}$ are parameters and $n$ is the number of hidden neurons in the
model. 

In this work, we study the
how the approximation properties of two-layer neural networks depends on the number of hidden neurons. In particular, we consider
the class of functions where the number of hidden neurons is bounded,
\begin{equation}\label{ShallowNN_bounded}
\Sigma^n_d(\sigma)=\left\{\displaystyle\sum_{i = 1}^n \beta_i\sigma(\omega_i \cdot x + b_i):
\omega_i\in\mathbb{R}^d,b_i,\beta_i\in\mathbb{R}\right\},
\end{equation}
and prove the Theorem \ref{approximation_rate_theorem} concerning the 
order of approximation as $n\rightarrow \infty$ for activation functions with
polynomial decay and Theorem \ref{periodic-activation} which applies to neural networks with periodic activation functions. Our results make the assumption that the function to be approximated, $f:\mathbb{R}^d\rightarrow \mathbb{R}$, has bounded Barron norm
\begin{equation}\label{barron-norm}
 \|f\|_{\mathcal{B}^s} = \int_{\mathbb{R}^d} (1 + |\omega|)^s|\hat f(\omega)|d\omega,
\end{equation}
and we consider the problem of approximating $f$ on a bounded domain $\Omega$. This is a significantly weaker assumption than the strong smoothness assumption made in \cite{petrushev1998approximation,kainen2007sobolev}.
Similar results appear in \cite{barron1993universal,hornik1994degree}, but we have improved their bound by a logarithmic factor
for exponentially decaying activation functions and generalized these result to polynomially decaying
activation functions. We also leverage this result to obtain a somewhat worse, though still dimension independent, approximation rate without the polynomial decay condition. This result ultimately applies to every activation function of bounded variation. Finally, we extend the stratified sampling argument in \cite{klusowski2016uniform} to polynomially decaying activation functions in Theorem \ref{stratified-sampling} and to periodic activation functions in Theorem \ref{periodic-improved}. This gives an improvement on the asymptotic rate of convergence under mild additional assumptions.

The paper is organized as follows. In the next section, we discuss some basic results concerning the Fourier transform.
We use these results to provide a simplified proof using Fourier analysis of the density result in \cite{leshno1993multilayer} under the mild
additional assumption of polynomial growth on $\sigma$. 
Then, in the third section, we study
the order of approximation and
prove Theorems \ref{approximation_rate_theorem} and \ref{periodic-activation}, extending the result in \cite{barron1993universal,hornik1994degree} to polynomially decaying activation functions, respectively periodic activation functions, and removing a logarithmic factor in the rate of approximation. In the fourth section, we provide a new argument using an approximate integral representation to obtain dimension independent results without the polynomial decay condition in Theorem \ref{theorem_no_decay}. In the fifth section, we use a stratified sampling argument to prove Theorems \ref{stratified-sampling} and \ref{periodic-improved}, which improve upon the convergence rates in Theorems \ref{approximation_rate_theorem} and \ref{periodic-activation} under mild additional assumptions. This generalizes the results in \cite{klusowski2016uniform} to more general activation functions. Finally, we give concluding remarks and further research directions in the
conclusion.

\section{Preliminaries}
Our arguments will make use of the theory of tempered distributions (see
\cite{strichartz2003guide,stein2016introduction} for an
introduction) and we begin by collecting some results of independent interest, which will also be important later. We 
begin by noting that an activation function $\sigma$ which satisfies a polynomial growth condition
$|\sigma(x)| \leq C(1 + |x|)^n$ for some constants $C$ and $n$ is a tempered distribution. As a result,
we make this assumption on our activation functions in the following theorems. We briefly note that
this condition is sufficient, but not necessary (for instance an integrable function need not satisfy a 
pointwise polynomial growth bound) for $\sigma$ to represent a tempered distribution.

 We begin by studying the convolution of $\sigma$ with a Gaussian mollifier. Let $\eta$ be a Gaussian mollifier
 \begin{equation}
  \eta(x) = \frac{1}{\sqrt{\pi}}e^{-x^2}.
 \end{equation}
Set $\eta_\epsilon=\frac{1}{\epsilon}\eta(\frac{x}{\epsilon})$. Then consider $\sigma_{\epsilon}$
\begin{equation}
\label{sigma-epsilon}
\sigma_{\epsilon}(x):=\sigma\ast{\eta_\epsilon}(x)=\int_{\mathbb{R}}\sigma(x-y){\eta_\epsilon}(y)dy
\end{equation}
for a given activation function $\sigma$.

It is clear that $\sigma_{\epsilon}\in C^\infty(\mathbb{R})$. Moreover, by considering the Fourier transform (as a tempered
distribution) we see that
\begin{equation}\label{eq_278}
 \hat{\sigma}_{\epsilon} = \hat{\sigma}\hat{\eta}_{\epsilon} = \sqrt{\pi}\epsilon\hat{\sigma}\eta_{\epsilon^{-1}}.
\end{equation} 
We begin by stating a lemma which characterizes the set of polynomials in terms of their
 Fourier transform.
\begin{lemma}\label{polynomial_lemma} Given a tempered distribution
  $\sigma$,  the following statements are equivalent:
\begin{enumerate}
\item $\sigma$ is a polynomial 
\item $\sigma_\epsilon$ given by \eqref{sigma-epsilon} is a polynomial for any
  $\epsilon>0$. 
\item $\text{\normalfont supp}(\hat{\sigma})\subset \{0\}$. 
\end{enumerate}
\end{lemma}
\begin{proof}
  We begin by proving that (3) and (1) are equivalent.
  This follows from a characterization of distributions supported at a
  single point (see \cite{strichartz2003guide}, section 6.3). In particular, a distribution supported
  at $0$ must be a finite linear combination of Dirac masses and their derivatives.
  In particular, if $\hat{\sigma}$ is supported at $0$, then
  \begin{equation}
   \hat{\sigma} = \displaystyle\sum_{i=1}^n a_i\delta^{(i)}.
  \end{equation}
  Taking the inverse Fourier transform and noting that the inverse Fourier transform of $\delta^{(i)}$ is $c_ix^i$,
  we see that $\sigma$ is a polynomial. This shows that (3) implies (1), for the converse we simply take the
  Fourier transform of a polynomial and note that it is a finite linear combination of Dirac masses and their derivatives.
  
  Finally, we prove the equivalence of (2) and (3). For this it suffices to show that $\hat{\sigma}$ is supported at $0$ iff
  $\hat{\sigma}_\epsilon$ is supported at $0$. This follows from equation \ref{eq_278} and the fact that
  $\eta_{\epsilon^{-1}}$ is nowhere vanishing.

\end{proof}

As an application of Lemma \ref{polynomial_lemma}, let us give a
simple proof of the following result.  The first proof of this result
can be found in \cite{leshno1993multilayer} and is summarized in
\cite{pinkus1999approximation}.  Extending this result to the case of
non-smooth activation is first done in several steps in
\cite{leshno1993multilayer}. Our contribution is to provide a much
simpler argument based on Fourier analysis.

\begin{theorem}\label{density_result}
Assume that $\sigma$ is a Riemann integrable function which satisfies a polynomial growth condition, i.e.
\begin{equation}
 |\sigma(t)|\leq C(1 + |t|)^p
\end{equation}
holds for some constants $C$ and $p$. Then if $\sigma$ is not a polynomial, $\Sigma_d(\sigma)$ in dense in
$C(\Omega)$ for any compact $\Omega\subset \mathbb{R}^n$.
\end{theorem}
\begin{proof}
Let us first prove the theorem in a special case that $\sigma\in
C^\infty(\mathbb{R})$.
Since $\sigma\in C^\infty(\mathbb{R})$, it follows that for every $\omega,b$
\begin{equation}
 \frac{\partial}{\partial \omega_j}\sigma(\omega\cdot x+b) = 
 \lim_{n\rightarrow \infty}\frac{\sigma((\omega+h e_j)\cdot x+b)-\sigma(\omega\cdot x+b)}{h} \in \overline{\Sigma}_d(\sigma)
\end{equation}
for all $j=1,...,d$. 

By the same argument, for $\alpha = (\alpha_1,...,\alpha_d)$
$$D^\alpha_\omega\sigma(\omega\cdot x+b)\in\overline{\Sigma}_d(\sigma)$$
for all $k\in\mathbb{N}$, $j=1,...,d$, $\omega\in\mathbb{R}^d$ and $b\in\mathbb{R}$.

Now 
$$D^\alpha_\omega\sigma(\omega\cdot x+b)=x^\alpha\sigma^{(k)}(\omega\cdot x+b),$$
where $x^\alpha = x_1^{\alpha_1}\cdots x_d^{\alpha_d}$.  Since
$\sigma$ is not a polynomial there exists a $\theta_k\in\mathbb{R}$
such that $\sigma^{(k)}(\theta_k)\ne0$.  Taking $\omega=0$ and
$b=\theta_k$, we thus see that $x_j^k\in\overline{\Sigma}_d(\sigma)$.
Thus, all polynomials of the form $x_1^{k_1}\cdots x_d^{k_d}$ are in
$\overline{\Sigma}_d(\sigma)$.

This implies that $\overline{\Sigma}_d(\sigma)$ contains all
polynomials.  By Weierstrass's Theorem \cite{stone1948generalized} it
follows that $\overline{\Sigma}_d(\sigma)$ contains $C(K)$ for each
compact $K\subset\mathbb{R}^n$. That is $\Sigma_d(\sigma)$ is dense in
$C(\mathbb{R}^d)$.

By the preceding lemma, it follows that $\Sigma_d(\sigma_\epsilon)$ is dense and so it suffices to 
show that $\sigma_{\epsilon}\in \Sigma_d(\sigma)$. This follows by using the 
Riemann integrability of $\sigma$ and approximating the integral
\begin{equation}
 \int_{\mathbb{R}}\sigma(x-y){\eta_\epsilon}(y)dy
\end{equation}
by a sequence of Riemann sums, each of which is clearly in $\Sigma_d(\sigma)$.
\end{proof}

\section{Convergence Rates in Sobolev Norms}
In this section, we study the order of approximation for two-layer neural networks as the number of neurons increases. 
In particular, we consider the
space of functions represented by a two-layer neural network with $n$ neurons and activation function $\sigma$ given in \eqref{ShallowNN_bounded},
and ask the following question: Given a function $f$ on a bounded domain, how many neurons do we need to approximate $f$
with a given accuracy?

Specifically, we will consider the problem of approximating a function $f$ with bounded Barron norm \eqref{barron-norm} in the Sobolev space $H^m(\Omega)$. Our first step will be to prove a lemma showing that the Sobolev norm is bounded by the Barron norm.

\begin{lemma}\label{smoothness-lemma}
 Let $m \geq 0$ be an integer and $\Omega\subset \mathbb{R}^d$ a bounded domain. Then for any Schwartz function $f$, we have
 \begin{equation}
  \|f\|_{H^m(\Omega)} \leq C(m)|\Omega|^{1/2}\|f\|_{\mathcal{B}^m}.
 \end{equation}

\end{lemma}
\begin{proof}
 Let $\chi_{\Omega}$ be the characteristic function of $\Omega$. Let $\alpha$ be any multi-index with $|\alpha|\leq m$. Then we have
 \begin{equation}
  \|D^\alpha f\|_{L^2(\Omega)} = \|\chi_{\Omega} D^\alpha f\|_{L^2(\mathbb{R}^d)} = \|\hat\chi_{\Omega} * \widehat{D^\alpha f}\|_{L^2(\mathbb{R}^d)}.
 \end{equation}
 Now we use Young's inequality and Plancherel's theorem to obtain
 \begin{equation}
  \|\hat\chi_\Omega * \widehat{D^\alpha f}\|_{L^2(\mathbb{R}^d)} \leq \|\hat\chi_\Omega\|_{L^2(\mathbb{R}^d)}\|\widehat{D^\alpha f}\|_{L^1(\mathbb{R}^d)} \leq \|\chi_{\Omega}\|_{L^2(\mathbb{R}^d)}\|f\|_{\mathcal{B}^m} = |\Omega|^{\frac{1}{2}}\|f\|_{\mathcal{B}^m}.
 \end{equation}
 Summing this over all multi-indices $\alpha\leq m$, we get
 \begin{equation}
  \|f\|_{H^m(\Omega)} \leq C(m)|\Omega|^{\frac{1}{2}}\|f\|_{\mathcal{B}^m},
 \end{equation}
 as desired.

\end{proof}

If we let $\mathcal{B}^m = \{f:\|f\|_{\mathcal{B}^m}  < \infty\}$ denote the Banach space of functions with bounded Barron norm, then this lemma implies that $\mathcal{B}^m \subset H^m(\Omega)$ for any bounded set $\Omega$, since the Schwartz functions are clearly dense in $\mathcal{B}^m$ (as the Barron norm is a polynomially weighted $L^1$ norm in Fourier space).

We now prove the following result, which shows that functions $f\in \mathcal{B}^{m+1}$ can be efficiently approximated by neural networks in the $H^m(\Omega)$ norm with polynomially decaying activation functions. A similar result can be found in \cite{barron1993universal},
but our result applies to non-sigmoidal activation functions $\sigma$. The class of functions we consider neither contains nor is contained by the class of sigmoidal functions. Compared with the result on exponentially decaying activation functions in \cite{hornik1994degree}, we extend the results to polynomially decaying activation functions and improve the rate of approximation by a logarithmic factor.

\begin{theorem}\label{approximation_rate_theorem}
 Let $\Omega\subset \mathbb{R}^d$ be a bounded domain. If the activation function $\sigma\in W^{m,\infty}(\mathbb{R})$ is non-zero and satisfies the polynomial decay condition 
 \begin{equation}\label{growth_condition}
  |\sigma^{(k)}(t)| \leq C_p(1 + |t|)^{-p}
 \end{equation}
 for $0\leq k\leq m$ and some $p > 1$, we have
 \begin{equation}
  \inf_{f_n\in \Sigma_d^n(\sigma)}\|f - f_n\|_{H^m(\Omega)} \leq |\Omega|^{\frac{1}{2}}C(p,m,\text{\normalfont diam}(\Omega),\sigma)n^{-\frac{1}{2}}\|f\|_{\mathcal{B}^{m+1}},
 \end{equation}
 for any $f\in \mathcal{B}^{m+1}$.

\end{theorem}
Before we proceed to the proof, we discuss how this bound depends on the dimension $d$. We first note that $|\Omega|$ may in a sense depend on the dimension, as the measure may be exponentially large in high dimensions. However, bounding the $H^m$ error over a larger set is also proportionally stronger. This can be seen by noting that dividing by the $|\Omega|^\frac{1}{2}$ factor transforms the left hand side from the total squared error to the average squared error.

The dimension dependence of this result is a consequence of how the Barron norm behaves in high dimensions. This issue is discussed in \cite{barron1993universal}, where the norm $\|\cdot\|_{\mathcal{B}^1}$ is analyzed for a number of different function classes. A particularly representative result found there is that $H^{\frac{d}{2}+2}\subset \mathcal{B}^1$. This shows that sufficiently smooth functions have bounded Barron norm, where the required number of derivatives depends upon the dimension. It is known that approximating functions with such a dimension dependent level of smoothness can be done efficiently \cite{petrushev1998approximation, kainen2007sobolev}. However, the Barron space $\mathcal{B}^1$ is significantly larger that $H^{\frac{d}{2}+2}$, in fact we only have $\mathcal{B}^1 \subset H^1$ by lemma \ref{smoothness-lemma}. The precise properties of the Barron norm in high dimensions are an interesting research direction which would help explain exactly how shallow neural networks help alleviate the curse of dimensionality.
\begin{proof}
 Note first that Lemma \ref{smoothness-lemma} implies that $f\in \mathcal{B}^{m+1}\subset \mathcal{B}^{m}\subset H^m(\Omega)$. We will later use this fact that $f$ is in the Hilbert space $H^m(\Omega)$.
 
 Note that the growth condition on $\sigma$ implies that $\sigma\in L^1(\mathbb{R})$ and thus the Fourier transform of
 $\sigma$ is well-defined and continuous. Since $\sigma$ is non-zero, this implies that $\hat{\sigma}(a)\neq 0$ for
 some $a\neq 0$. Via a change of variables, this means that for all $x$ and $\omega$, we have
 \begin{equation}
  0\neq \hat{\sigma}(a)=\frac{1}{2\pi}\int_{\mathbb{R}}\sigma(\omega\cdot x+b)e^{-ia(\omega\cdot x+b)}db,
 \end{equation}
 and so
 \begin{equation}
  e^{ia\omega \cdot x} = \frac{1}{2\pi\hat{\sigma}(a)}\int_{\mathbb{R}}\sigma(\omega\cdot x+b)e^{-iab}db.
 \end{equation}
 Likewise, since the growth condition also implies that $\sigma^{(k)}\in L^1$, we can differentiate the above expression 
 under the integral with respect to $x$.

 This allows us to write the Fourier mode $e^{ia\omega \cdot x}$ as an integral of neuron output functions. We substitute this
 into the Fourier representation of $f$
 (note that the assumption we make implies that $\hat{f}\in L^1$ so this
 is rigorously justified for a.e. $x$) to get
 \begin{equation}\label{integral_representation}
  f(x) = \int_{\mathbb{R}^d} e^{i\omega\cdot x}\hat{f}(\omega)d\omega = 
  \int_{\mathbb{R}^d}\int_\mathbb{R}\frac{1}{2\pi\hat{\sigma}(a)}
  \sigma\left(\frac{\omega}{a}\cdot x+b\right)\hat{f}(\omega)e^{-iab}dbd\omega.
 \end{equation}
 
 Given this integral representation, we follow a similar line of reasoning as in \cite{barron1993universal}. Our argument differs from previous arguments in how we write $f$ as convex combination of shifts and dilations of $\sigma$. This is what allows us to relax our assumptions on $\sigma$. In order to do this, we must first find a way to normalize the above integral.
 
 The above integral is on an unbounded domain, but the decay assumption on the Fourier transform of $f$ allows us to normalize the integral in the $\omega$ direction.
 To normalize the integral in the $b$ direction, we must use the assumption that $x$ is bounded and that $\sigma$ decays polynomially. Consider the $\sigma$ part of the above integral representation,
 \begin{equation}
  D_x^\alpha\sigma\left(\frac{\omega}{a}\cdot x+b\right).
 \end{equation}
 Note that by the triangle inequality and the boundedness of $x\in \Omega$, we can obtain a lower bound on the above argument uniformly in $x$. Specifically, we have
 \begin{equation}\label{eq_828}
  \left|\frac{\omega}{a}\cdot x+b\right| \geq \max\left(0,|b| - \frac{R|\omega|}{|a|}\right),
 \end{equation}
 where $R$ is the maximum norm of an element of $\Omega$. Note that without loss of generality, we can translate $\Omega$ so that it contains the origin and so $R \leq  \text{\normalfont diam}(\Omega)$.
 
 Combining this with the polynomial decay of $\omega$ \eqref{growth_condition} implies that
 \begin{equation}\label{eq_779}
  \left|\sigma^{(k)}\left(\frac{\omega}{a}\cdot x+b\right)\right| \leq C_p\left(1 + \left|\frac{\omega}{a}\cdot x+b\right|\right)^{-p} \leq C_p\left(1 + \max\left(0,|b| - \frac{R|\omega|}{|a|}\right)\right)^{-p}.
 \end{equation}
Thus the function $h$ defined by 
 \begin{equation}\label{h_definition}
  h(b,\omega) = \left(1 + \max\left(0,|b| - \frac{R|\omega|}{|a|}\right)\right)^{-p}
 \end{equation}
 provides (up to a constant) an upper bound on $\sigma^{(k)}\left(\frac{\omega}{a}\cdot x+b\right)$ uniformly in $x$.
 The decay rate of $h$ is fast enough to make it integrable in $b$. Moreover, its integral in $b$ grows at most linearly with $\omega$. Namely, we calculate
 \begin{equation}\label{eq_775}
 \begin{split}
  \int_\mathbb{R} h(b,\omega)db& = \int_{|b|\leq \frac{R|\omega|}{|a|}} db + 2\int_{b > \frac{R|\omega|}{|a|}} \left(1 + b - \frac{R|\omega|}{|a|}\right)^{-p}db \\
  & =~2R|a|^{-1}|\omega| + 2\left[(1-p)^{-1}\left(1 + b - \frac{R|\omega|}{|a|}\right)^{1-p}\right]_{\frac{R|\omega|}{|a|}}^\infty \\
  &=~2R|a|^{-1}|\omega| + \frac{2}{p-1}\leq C_1(p,\text{\normalfont diam}(\Omega),\sigma) (1 + |\omega|).
  \end{split}
 \end{equation}
 Combining \eqref{eq_775} with our assumption on the Fourier transform, we get
 \begin{equation}\label{bound_on_I}
  I(p,\Omega,\sigma,f) = \int_{\mathbb{R}^d}\int_\mathbb{R} (1 + |\omega|)^m h(b,\omega) |\hat{f}(\omega)| dbd\omega
  \leq C_1(p,\text{\normalfont diam}(\Omega),\sigma)\|f\|_{\mathcal{B}^{m+1}}.
 \end{equation}
We use this to introduce a probability measure $\lambda$ on $\mathbb{R}^{d+1} = (\omega,b)$ given by
 \begin{equation}
  d\lambda = \frac{1}{I(p,\Omega,\sigma,f)}(1 + |\omega|)^m h(b,\omega) |\hat{f}(\omega)| dbd\omega,
 \end{equation}
 which allows us to write (using the integral representation \eqref{integral_representation})
 \begin{equation}
  f(x) = \mathbb{E}_{d\lambda}\left(J(\omega,b)e^{i\theta(\omega,b)}\sigma\left(\frac{\omega}{a}\cdot x+b\right)\right),
 \end{equation}
where 
$$
\theta(\omega,b)=\theta(\hat{f}(\omega)) - \theta(\hat{\sigma}(a)) - ab
$$
and 
\begin{equation}\label{definition-of-J}
J(\omega,b)= (2\pi |\hat{\sigma}(a)|)^{-1}I(p,\Omega,\sigma,f)(1 + |\omega|)^{-m} h(b,\omega)^{-1}.
\end{equation}
Since $f(x)$ is real, we can replace the complex number $e^{i\theta(\omega,b)}$ by its real part, which we call $\chi(\omega,b)\in[-1,1]$, to get
 \begin{equation}\label{eq_837}
  f(x) = \mathbb{E}_{d\lambda}\left(J(\omega,b)\chi(\omega,b)\sigma\left(\frac{\omega}{a}\cdot x+b\right)\right).
 \end{equation}
 This writes $f\in H^m(\Omega)$ as an infinite convex combination of the functions $f_{\omega,b}\in H^m(\Omega)$, where
 \begin{equation}\label{eq_887}
  f_{\omega,b}(x) = J(\omega,b)\chi(\omega,b)\sigma\left(\frac{\omega}{a}\cdot x+b\right).
 \end{equation}
 We now use Lemma 1 from \cite{barron1993universal} to conclude that for each $n$ there exists an $f_n$ which is a convex combination of at most $n$ distinct $f_{\omega,b}$, and thus $f_n\in \Sigma_d^n(\sigma)$, such that
 \begin{equation}\label{eq_891}
  \|f-f_n\|_{H^m(\Omega)} \leq Cn^{-\frac{1}{2}},
 \end{equation}
 where $C = \sup_{\omega,b}\|f_{\omega,b}\|_{H^m(\Omega)}$.

We proceed to estimate $\sup_{\omega,b} \|f_{\omega,b}\|_{H^m(\Omega)}$. Since $\Omega$ is bounded, it has finite measure, and Cauchy-Schwartz implies that
\begin{equation}\label{eq_920}
 \|f_{\omega,b}\|_{H^m(\Omega)} \leq |\Omega|^{\frac{1}{2}}\|f_{\omega,b}\|_{W^{m,\infty}(\Omega)},
\end{equation}
and so it suffices to bound $\|D_x^\alpha f_{\omega,b}\|_{L^\infty(\Omega)}$ for each $|\alpha|\leq m$.

Expanding the definition \eqref{eq_887} of $f_{\omega,b}$, recalling that $\chi(\omega,b)\in [-1,1]$, and recalling the definition \eqref{definition-of-J} of $J(\omega,b)$ , we get
\begin{equation}\label{eq_926}
\begin{split}
 \|D_x^\alpha f_{\omega,b}\|_{L^\infty(\Omega)} &\leq \left\|J(\omega,b)D_x^\alpha\sigma\left(a^{-1}\omega\cdot x+b\right)\right\|_{L^\infty(\Omega)}
  \\
  &\leq (2\pi |a^{|\alpha|}\hat{\sigma}(a)|)^{-1}I(p,\Omega,\sigma,f)(1+|\omega|)^{-m}\left\|h(b,\omega)^{-1}D_x^\alpha\sigma\left(a^{-1}\omega\cdot x+b\right)\right\|_{L^\infty(\Omega)}.
  \end{split}
\end{equation}
 Since $|\alpha|\leq m$ and $\sigma\in W^{m,\infty}$, we obtain 
 $$\left|D_x^\alpha\sigma\left(a^{-1}\omega\cdot x+b\right)\right|\leq |a|^{-|\alpha|}(1 + |\omega|)^m\sigma^{(|\alpha|)}(a^{-1}\omega\cdot x+b)|,$$
and so, plugging this into \eqref{eq_926}, we get
 \begin{equation}\label{equation_558}
  \|D_x^\alpha f_{\omega,b}\|_{L^\infty(\Omega)}
  \leq (2\pi |a^{|\alpha|}\hat{\sigma}(a)|)^{-1}I(p,\Omega,\sigma,f)\left\|h(b,\omega)^{-1}\sigma^{(|\alpha|)}\left(a^{-1}\omega\cdot x+b\right)\right\|_{L^\infty(\Omega)}.
 \end{equation}
 Utilizing the fact \eqref{eq_779} that $h(b,\omega)$ upper bounds $\sigma^{(k)}\left(a^{-1}\omega\cdot x+b\right)$ uniformly in $x\in \Omega$, we obtain
 \begin{equation}
  \left\|h(b,\omega)^{-1}\sigma^{(|\alpha|)}\left(a^{-1}\omega\cdot x+b\right)\right\|_{L^\infty(\omega,b)} \leq C_p.
 \end{equation}
 Plugging this into \eqref{equation_558}, summing over $\alpha\leq m$, and using \eqref{eq_920}, we obtain
 \begin{equation}\label{eq_890}
  \sup_{\omega,b}\|f_{\omega,b}\|_{H^m(\Omega)} \leq |\Omega|^{\frac{1}{2}}(2\pi |\hat{\sigma}(a)|)^{-1}I(p,\Omega,\sigma,f)C_p \displaystyle\sum_{|\alpha|\leq m}|a|^{-|\alpha|}
 \end{equation}
 Finally, utilizing the bound \eqref{bound_on_I} and equation \eqref{eq_891}, we obtain 
 \begin{equation}
  \inf_{f_n\in \Sigma_d^n(\sigma)}\|f - f_n\|_{H^m(\Omega)} \leq |\Omega|^{\frac{1}{2}}C(p,m,\text{\normalfont diam}(\Omega),\sigma)\|f\|_{\mathcal{B}^{m+1}}n^{-\frac{1}{2}}
 \end{equation}
 as desired.
 \end{proof}
 
 Finally, we note that the approximation rate in this theorem holds as long as the growth condition \eqref{growth_condition}
 hold for some $f\in \Sigma_d(\sigma)$, i.e. the condition \eqref{growth_condition} need not hold for $\sigma$ itself.
 We state this as a corollary below.
 \begin{corollary}\label{GeneralApproximation}
  Let $\sigma\in W^{m,\infty}_{loc}(\mathbb{R})$ be an activation function and suppose that there exists a $\nu\in \Sigma_1^{n_0}(\sigma)$ which satisfies the polynomial decay condition \eqref{growth_condition} in Theorem \ref{approximation_rate_theorem}. Then for any $f$ satisfying the assumptions of Theorem \ref{approximation_rate_theorem}, we have
  \begin{equation}
     \inf_{f_n\in \Sigma_d^n(\sigma)}\|f - f_n\|_{H^m(\Omega)} \leq |\Omega|^{\frac{1}{2}}C(p,m,\text{\normalfont diam}(\Omega),\sigma)\sqrt{n_0}\|f\|_{\mathcal{B}^{m+1}}n^{-\frac{1}{2}}.
  \end{equation}

 \end{corollary}
 \begin{proof}
 The result follows immediately from Theorem \ref{approximation_rate_theorem} and the observation that $v\in \Sigma_1^{n_0}(\sigma)$ implies that
 \begin{equation}
  \Sigma_d^n(\nu) \subset \Sigma_d^{nn_0}(\sigma).
 \end{equation}

 \end{proof}
 
This includes many popular activation functions, such as the rectified linear units \cite{nair2010rectified} and logistic sigmoid activation functions. Below we provide a table listing some well-known activation functions to which this theorem applies.
\begin{center}
\begin{tabular}{ |c|c|c|c|c| } 
 \hline
 Activation Function & $\sigma(x)$ & Maximal $m$ & $n_0$ & $\nu(x)$ \\
 \hline
 Sigmoidal (Logistic) & $(1 + e^{-x})^{-1}$ & $\infty$ & $2$ & $\sigma(x+1) - \sigma(x)$ \\
 \hline

 Arctan & $\arctan(x)$ & $\infty$ & $2$ & $\sigma(x+1) - \sigma(x)$ \\ 
 \hline
 Hyperbolic Tangent & $\tanh(x)$ & $\infty$ & $2$ & $\sigma(x+1) - \sigma(x)$ \\
 \hline
 SoftPlus \cite{glorot2011deep} & $\log(1 + e^x)$ & $\infty$ & $3$ & $\sigma(x+1) + \sigma(x - 1) - 2\sigma(x)$ \\
 \hline
 ReLU\cite{nair2010rectified} & $\max(0,x)$ & $1$ & $3$ & $\sigma(x+1) + \sigma(x - 1) - 2\sigma(x)$ \\
 \hline
  Leaky ReLU\cite{maas2013rectifier} & $\epsilon x + (1-\epsilon)\max(0,x)$ & $1$ & $3$ & $\sigma(x+1) + \sigma(x - 1) - 2\sigma(x)$ \\
 \hline
 $k$-th power of ReLU & $[\max(0,x)]^k$ & $k$ & $k+2$ & $\sum_{i=0}^{k+1}(-1)^i\binom{k+1}{i} \sigma(x - \lfloor \frac{k+1}{2} \rfloor + i)$\\
 \hline
\end{tabular}
\end{center}

Next, we consider the case of periodic activation functions. We show that neural networks with periodic activation functions achieve the same rate of approximation in Theorem \ref{approximation_rate_theorem}. The argument makes use of a modified integral representation and allows us to relax the smoothness condition on $f$, which now only has to be in $\mathcal{B}^m$.
\begin{theorem}\label{periodic-activation}
 Let $\Omega\subset \mathbb{R}^d$ be a bounded domain. If the activation function $\sigma\in W^{m,\infty}(\mathbb{R})$ is a non-constant periodic function, we have
 \begin{equation}
  \inf_{f_n\in \Sigma_d^n(\sigma)}\|f - f_n\|_{H^m(\Omega)} \leq |\Omega|^{\frac{1}{2}}C(\sigma)n^{-\frac{1}{2}}\|f\|_{\mathcal{B}^m},
 \end{equation}
 for any $f\in \mathcal{B}^{m}$.
\end{theorem}
\begin{proof}
 By dilating $\sigma$ if necessary, we may assume without loss of generality that $\sigma$ is periodic on $[0,2\pi]$. Consider the Fourier series of $\sigma$
 \begin{equation}
  \sigma(x) = \displaystyle\sum_{i=-\infty}^\infty a_i e^{ix},
 \end{equation}
 with coefficients
 \begin{equation}\label{eq_1027}
  a_i = \frac{1}{2\pi}\int_0^{2\pi} \sigma(b)e^{-ib}db. 
 \end{equation}
 The assumption that $\sigma$ is non-constant means that there exists some $i$ such that $a_i \neq 0$. Note that we do not need the Fourier series to converge pointwise to $\sigma$, all we need is for some $a_i$ to be non-zero and the integrals in \eqref{eq_1027} to converge (which is does since $\sigma\in W^{m,\infty}$). Notice that shifting $\sigma$ by $t$, i.e. replacing $\sigma$ by $\sigma(\cdot+t)$, scales the coefficient $a_i$ by $e^{it}$. Setting $t = (\omega \cdot x)$, we get
 \begin{equation}
  e^{i\omega\cdot x} = \frac{1}{2\pi a_i}\int_0^{2\pi} \sigma\left(\omega\cdot x + b\right)e^{-ib}db.
 \end{equation}
 Plugging this into the Fourier representation of $f$, we see that
 \begin{equation}
  f(x) = \int_{\mathbb{R}^d} e^{i\omega\cdot x}\hat{f}(\omega)d\omega = \frac{1}{2\pi a_i}
  \int_{\mathbb{R}^d}\int_0^{2\pi}\sigma\left(\omega\cdot x + b\right)e^{-ib}\hat{f}(\omega)dbd\omega.
 \end{equation}
 Since $f(x)$ is real, we can add this to its conjugate to obtain the representation
 \begin{equation}\label{eq_1029}
  f(x) = \int_{\mathbb{R}^d} e^{i\omega\cdot x}\hat{f}(\omega)d\omega = \frac{1}{2\pi |a_i|}
  \int_{\mathbb{R}^d}\int_0^{2\pi}\sigma\left(\omega\cdot x + b\right)\chi(\omega,b)|\hat{f}(\omega)|dbd\omega,
 \end{equation}
 where $|\chi(\omega,b)|\leq 1$.

 Using the integrability condition on $\hat{f}$, we define the probability distribution on $\mathbb{R}^d \times [0,2\pi]$ as
 \begin{equation}
  d\lambda = \frac{1}{2\pi\|f\|_{\mathcal{B}^m}}(1 + |\omega|)^m|\hat{f}(x)|dxdb.
 \end{equation}
 Then equation \eqref{eq_1029} becomes
 \begin{equation}\label{periodic_representation}
 f(x) = \mathbb{E}_{d\lambda}\left(\|f\|_{\mathcal{B}^m}|a_i|^{-1}(1 + |\omega|)^{-m}\chi(\omega,b)\sigma\left(\omega\cdot x + b\right)\right).
 \end{equation}
 We have now written $f\in \mathcal{B}^m\subset H^m(\Omega)$ as a convex combination of functions $f_{\omega,b}\in H^m(\Omega)$. As in the proof of \ref{approximation_rate_theorem}, we now utilize Lemma 1 in \cite{barron1993universal} and proceed to bound $\|f_{\omega,b}\|_{H^m(\Omega)}$ using much the same argument.

\end{proof}

\section{Activation Functions without Decay}
 In this section, we show that one can remove the decay condition on $\sigma$, but this results in a slightly worse (though still dimension independent) approximation rate. The main new tool here is the use of an approximate integral representation, followed by an optimization over the accuracy of the representation. Finally, as a corollary we are able to obtain a dimension independent approximation rate for all bounded, integrable activation functions and all activation functions of bounded variation.
 
\begin{theorem}\label{theorem_no_decay}
 Let $\Omega\subset \mathbb{R}^d$ be a bounded domain. Suppose that $\sigma\in L^\infty(\mathbb{R})$ and that there exists an open interval $I$ such that $\hat\sigma$ (as a tempered distribution) is a non-zero bounded function on $I$.
Then we have
 \begin{equation}
  \inf_{f_n\in \Sigma_d^n(\sigma)}\|f - f_n\|_{L^2(\Omega)} \leq |\Omega|^{\frac{1}{2}}C(\text{\normalfont diam}(\Omega),\sigma)n^{-\frac{1}{4}}\|f\|_{\mathcal{B}^1},
 \end{equation}
 for any $f\in \mathcal{B}^{1}$.
\end{theorem}

\begin{proof}
 Let $0\neq a\in I$ be a Lesbesgue point of $\hat\sigma$ for which $\hat\sigma(a)\neq 0$. Note that $\hat\sigma$ is bounded on $I$ and thus locally integrable. Such a point must exist since the set of Lesbesgue points has full measure and $\hat{\sigma}$ is non-zero on $I$. Let $\phi$ be a Schwartz function whose Fourier transform $\hat\phi$ is supported in $[-1,1]$, $0\leq \hat\phi \leq 1$, and $\hat\phi = 1$ on $[-\frac{1}{2}, \frac{1}{2}]$. For a fixed $0 < \epsilon$ such that $(a+\epsilon, a-\epsilon) \subset I$, we note the identity
 \begin{equation}\label{fourier_identity_1}
 \frac{1}{2\pi}\int_{\mathbb{R}}\sigma(\omega\cdot x+b)\phi(\epsilon b)e^{-iab}db = e^{ia\omega\cdot x}\frac{1}{\epsilon} \int_{-\epsilon}^{\epsilon} \hat\sigma(a+t)\hat\phi\left(\frac{t}{\epsilon}\right)e^{i(\omega\cdot x)t}dt.
 \end{equation}
 Here the equality comes from the assumption that $\hat{\sigma}$ is a bounded function on $(a-\epsilon,a+\epsilon)\subset I$ (i.e. that Plancherel's theorem holds using bona-fide integrals for Schwartz functions whose Fourier transforms are supported in $I$) and the observation (or calculation) that
 $$\mathcal{F}(\phi(\epsilon b))(t) = \frac{1}{\epsilon}\hat{\phi}\left(\frac{t}{\epsilon}\right)$$
 and
 $$\mathcal{F}(\sigma(\omega\cdot x+b)e^{-iab})(t) = e^{i(t+a)\omega\cdot x}\hat\sigma(t + a).$$
 Define
 \begin{equation}
  C(\epsilon) = \frac{1}{\epsilon} \int_{-\epsilon}^{\epsilon} \hat\sigma(a+t)\hat\phi\left(\frac{t}{\epsilon}\right)dt.
 \end{equation}
 Since $a$ is a Lesbesgue point, we have
 \begin{equation}
  \lim_{\epsilon\rightarrow 0} C(\epsilon) = \sigma(a)\int_{-1}^1 \hat\phi(x) dx \neq 0.
 \end{equation}
 Thus, for sufficiently small $\epsilon$, $C(\epsilon)$ is bounded away from $0$.
 In addition, because $|e^{i(\omega\cdot x)t} - 1| \leq |(\omega\cdot x)t|$, we see that, since $0\leq \hat{\phi}\leq 1$,
 \begin{equation}
  \left|C(\epsilon) - \frac{1}{\epsilon} \int_{-\epsilon}^{\epsilon} \hat\sigma(a+t)\hat\phi\left(\frac{t}{\epsilon}\right)e^{i(\omega\cdot x)t}dt\right| \leq \frac{1}{\epsilon} \int_{-\epsilon}^{\epsilon} \left|\hat\sigma(a+t)\hat\phi\left(\frac{t}{\epsilon}\right)\right||(\omega\cdot x)t|dt \leq 2\epsilon\|\hat\sigma\|_{L^\infty(I)} |\omega \cdot x|.
 \end{equation}
  Dividing by $C(\epsilon)$ (which is bounded away from $0$ for small enough $\epsilon$) and multiplying by $e^{ia\omega\cdot x}$ (which is bounded), we get (using the identity \eqref{fourier_identity_1})
 \begin{equation}
  \left|e^{ia\omega\cdot x} - \frac{1}{2\pi C(\epsilon)}\int_{\mathbb{R}}\sigma(\omega\cdot x+b)\phi(\epsilon b)e^{-iab}db\right| \lesssim \epsilon|\omega\cdot x|,
 \end{equation}
 where the implied constant depends upon $\sigma$, but not upon $\epsilon$ as long as $\epsilon$ is sufficiently small.
This gives us the approximate integral representation
\begin{equation}
 \frac{1}{2\pi C(\epsilon)}\int_{\mathbb{R}}\sigma(\omega\cdot x+b)\phi(\epsilon b)e^{-iab}db = e^{ia\omega\cdot x} + O(\epsilon |\omega\cdot x|).
\end{equation}
Given that $x\in \Omega$ and $\Omega$ is bounded, we obtain, by allowing implied constant in the $O$ notation to depend on $\Omega$ (specifically, since we can translate $\Omega$ without loss of generality so it contains the origin, $\text{\normalfont diam}(\Omega)$),
\begin{equation}
 \frac{1}{2\pi C(\epsilon)}\int_{\mathbb{R}}\sigma(\omega\cdot x+b)\phi(\epsilon b)e^{-iab}db = e^{ia\omega\cdot x} + O(\epsilon |\omega|).
\end{equation}
We now use this to construct the following approximate integral representation of $f$.
\begin{equation}
   f(x) = \int_{\mathbb{R}^d} e^{i\omega\cdot x}\hat{f}(\omega)d\omega = \frac{1}{2\pi C(\epsilon)}\int_{\mathbb{R}^d} \int_{\mathbb{R}} \hat{f}(\omega) \sigma\left(\frac{\omega}{a}\cdot x + b\right)\phi(\epsilon b) e^{-iab} dbd\omega + \int_{\mathbb{R}^d}O(\epsilon|\hat f(\omega)||\omega|)d\omega.
\end{equation}
Introducing the probability distribution
\begin{equation}
 d\lambda = \frac{|\hat{f}(\omega)||\phi(\epsilon b)|}{\|\hat{f}(\omega)\|_{L^1(\omega)}\|\phi(\epsilon b)\|_{L^1(b)}}d\omega db,
\end{equation}
we can rewrite this as
\begin{equation}
 f(x) = \mathbb{E}_{d\lambda}\left(\frac{\|\hat{f}(\omega)\|_{L^1(\omega)}\|\phi(\epsilon b)\|_{L^1(b)}}{2\pi |C(\epsilon)|}\frac{|C(\epsilon)|}{C(\epsilon)}\frac{\hat{f}(\omega)}{|\hat{f}(\omega)|}\frac{\phi(\epsilon b)}{|\phi(\epsilon b)|}e^{-iab}\sigma\left(\frac{\omega}{a}\cdot x + b\right)\right) + \int_{\mathbb{R}^d}O(\epsilon|\hat f(\omega)||\omega|)d\omega.
\end{equation}
The error term (i.e. second term) in this representation is bounded by $O(\epsilon \|f\|_{B^1})$. The first term is a convex combination of bounded functions and can be analyzed in the same way as in the proof of Theorem \ref{approximation_rate_theorem}. In particular, we use Lemma 1 in \cite{barron1993universal} combined with the assumption that $\sigma\in L^\infty$ to see that
\begin{equation}
 \inf_{f_n\in \Sigma_d^n(\sigma)}\|f - f_n\|_{L^2(\Omega)} \leq |\Omega|^{\frac{1}{2}}\frac{\|\hat{f}(\omega)\|_{L^1(\omega)}\|\phi(\epsilon b)\|_{L^1(b)}\|\sigma\|_{L^\infty}}{2\pi |C(\epsilon)|} n^{-\frac{1}{2}} + O(\epsilon\|f\|_{\mathcal{B}^1}).
\end{equation}
We now use the fact that $\|\phi(\epsilon b)\|_{L^1(db)} = O(\frac{1}{\epsilon})$, that $\|\hat{f}(\omega)\|_{L^1(\omega)} = \|f\|_{\mathcal{B}^0}$, and that $C(\epsilon)$ is bounded away from $0$, to obtain
 \begin{equation}
  \inf_{f_n\in \Sigma_d^n(\sigma)}\|f - f_n\|_{L^2(\Omega)} \lesssim |\Omega|^{\frac{1}{2}}\frac{\|f\|_{\mathcal{B}^0}}{\epsilon} n^{-\frac{1}{2}} + \epsilon\|f\|_{\mathcal{B}^1} \leq |\Omega|^{\frac{1}{2}}\|f\|_{\mathcal{B}^1}\left(\frac{1}{\epsilon}n^{-\frac{1}{2}} + \epsilon\right).
 \end{equation}
Optimizing over $\epsilon$, i.e. setting $\epsilon = n^{-1/4}$, we obtain
\begin{equation}
 \inf_{f_n\in \Sigma_d^n(\sigma)}\|f - f_n\|_{L^2(\Omega)} \lesssim |\Omega|^{\frac{1}{2}}\|f\|_{\mathcal{B}^1}n^{-\frac{1}{4}}
\end{equation}
as desired. Specifically, the $\lesssim$ notation used throughout the proof hides a constant which depends on $\sigma$, $\text{\normalfont diam}(\Omega)$ and the fixed function $\phi$.
\end{proof}

The assumption made on the activation function in Theorem \ref{theorem_no_decay} is quite technical. However, we note that it also appears to be extremely weak. In particular, the only bounded functions we have been able to find which don't satisfy the assumption are quasi-periodic functions of the form
\begin{equation}
 \displaystyle\sum_{i=0}^\infty a_ie^{if_ix},
\end{equation}
for a sequence of frequencies $f_i$ and coefficients satisfying $\sum|a_i|\leq \infty$. To add to this point, we provide the following corollary observing that it holds for all integrable activation functions $\sigma\in L^1(\mathbb{R})\cap L^\infty(\mathbb{R})$.

\begin{corollary}
 Suppose that activation function $\sigma\in L^1(\mathbb{R})\cap L^\infty(\mathbb{R})$ is non-zero. Then we have
 \begin{equation}
  \inf_{f_n\in \Sigma_d^n(\sigma)}\|f - f_n\|_{L^2(\Omega)} \leq |\Omega|^{\frac{1}{2}}C(\text{\normalfont diam}(\Omega),\sigma)n^{-\frac{1}{4}}\|f\|_{\mathcal{B}^1},
 \end{equation}
 for any $f\in \mathcal{B}^{1}$.
\end{corollary}

\begin{proof}
 If $0\neq \sigma\in L^1(\mathbb{R})$, then it is well-known that $\hat{\sigma}$ is a non-zero continuous function. Thus $\sigma$ satisfies the assumptions of Theorem \ref{theorem_no_decay} and we obtain the desired result.
\end{proof}

We also note that an extension similar to corollary \ref{GeneralApproximation} holds for this result as well. This allows us to extend this approximation rate to all activation functions which are of bounded variation, as the following corollary states.

\begin{corollary}
 Suppose that the activation function $\sigma\in BV(\mathbb{R})$ is non-constant. Then we have
  \begin{equation}
  \inf_{f_n\in \Sigma_d^n(\sigma)}\|f - f_n\|_{L^2(\Omega)} \leq |\Omega|^{\frac{1}{2}}C(\text{\normalfont diam}(\Omega),\sigma)n^{-\frac{1}{4}}\|f\|_{\mathcal{B}^1},
 \end{equation}
 for any $f\in \mathcal{B}^{1}$.
\end{corollary}
\begin{proof}
 We will show that under the given assumptions, the function $\tau(x) = \sigma(x+1) - \sigma(x)$, which is clearly non-zero since $\sigma$ is non-constant, is in $L^1(\mathbb{R})\cap L^\infty(\mathbb{R})$. This completes the proof in light of the previous corollary and discussion.
 
 First, it easily follows from the definition of the $\|\cdot\|_{BV}$ norm that $\|\tau\|_{\infty}\leq \|\sigma\|_{BV}$ and thus $\tau\in L^\infty(\mathbb{R})$. In addition, we have
 \begin{equation}
  \int_{-M}^M |\tau(x)|dx = \int_{-M}^M |\sigma(x+1) - \sigma(x)| dx = \int_{0}^1 \displaystyle\sum_{i=-M}^{M-1} |\sigma(x+i+1) - \sigma(x+i)|dx \leq \int_{0}^1 \|\sigma\|_{BV} dx = \|\sigma\|_{BV},
 \end{equation}
for any integer $M$. Taking the limit as $M\rightarrow \infty$, we get $\|\tau\|_{L^1}\leq \|\sigma\|_{BV}$ and thus $\tau \in L^1(\mathbb{R})$ as desired.
\end{proof}

 \section{An Improved Estimate Using Stratified Sampling}
 In this section, we show how the argument in the proof of Theorem \ref{approximation_rate_theorem} can be improved using a stratified sampling method. Our argument is based on the method presented in \cite{klusowski2016uniform,makovoz1996random}. This method allows us to obtain an improved asymtotic convergence rate under additional smoothness assumptions on both the activation function and the function to be approximated.
 \begin{theorem}\label{stratified-sampling}
 Let $\Omega\subset \mathbb{R}^d$ be a bounded domain and $\epsilon > 0$. If the activation function $\sigma\in W^{m,\infty}(\mathbb{R})$ is non-zero and it satisfies the polynomial decay condition
 \begin{equation}\label{growth_condition_two}
  |\sigma^{(k)}(s)| \leq C_p(1 + |s|)^{-p}
 \end{equation}
 for $0\leq k\leq m+1$ and some $p > 1$, we have
 \begin{equation}
  \inf_{f_n\in \Sigma_d^n(\sigma)}\|f - f_n\|_{H^m(\Omega)} \leq |\Omega|^{\frac{1}{2}}C(p,m,\text{\normalfont diam}(\Omega),\sigma)n^{-\frac{1}{2}-\frac{t}{(2+t)(d+1)}}\|f\|_{\mathcal{B}^{m+1+\epsilon}}
 \end{equation}
 where $t = \min(p-1,\epsilon)$, for any $f\in \mathcal{B}^{m+1+\epsilon}$.
 \end{theorem}
 
 \begin{proof}
  We follow the proof of Theorem \ref{approximation_rate_theorem} up to the integral representation of $D^\alpha f(x)$ given in equation \eqref{eq_837},
  \begin{equation}
    f(x) = \mathbb{E}_{d\lambda}\left(J(\omega,b)\chi(\omega,b)\sigma\left(\frac{\omega}{a}\cdot x+b\right)\right),   
  \end{equation}
  where $\chi(\omega,b)\in [-1,1]$, $d\lambda$ is given by
  \begin{equation}
   d\lambda = \frac{1}{I(p,\Omega,\sigma,f)}(1 + |\omega|)^m h(b,\omega) |\hat{f}(\omega)| dbd\omega,
  \end{equation}
  and $J(\omega,b)$ is given by equation \eqref{definition-of-J}.
  
 We proceed by removing $\chi(\omega,b)$ from this representation. We do this by introducing a random sign $\eta\in\{\pm1\}$, whose distribution depends on $\omega$ and $b$, such that the conditional distribution satisfies $\mathbb{E}(\eta|\omega,b) = \chi(\omega,b)$. To be more precise, this leads to a probability measure $d\tilde\lambda$ on $\mathbb{R}\times\mathbb{R}^d\times \{\pm1\}$ defined by
 \begin{equation}
  \tilde\lambda(B) = \int_{B_{+1}} \frac{\chi(\omega,b) + 1}{2} d\lambda + \int_{B_{-1}}\frac{1 - \chi(\omega,b)}{2}d\lambda,
 \end{equation}
 where $B_{+1},B_{-1}\subset \mathbb{R}\times\mathbb{R}^d$ are the intersections $B\cap (\mathbb{R}\times\mathbb{R}^d\times \{+1\})$ and $B\cap (\mathbb{R}\times\mathbb{R}^d\times \{-1\})$, respectively. We note that by construction this new distribution $d\tilde\lambda$ satisfies
 \begin{equation}
  f(x) = \mathbb{E}_{d\tilde\lambda}\left(\eta J(\omega,b)\sigma\left(\frac{\omega}{a}\cdot x+b\right)\right).   
  \end{equation}
 
 We now introduce a quantity $A \geq 1$ to be chosen later, and consider the set $S_A = \{(\omega,b,\eta):|b| \leq A, |\omega| \leq \frac{A|a|}{2R}\}$. Following the argument in \cite{klusowski2016uniform}, we partition the set $S_A$ into sets $S_1,...,S_{2n}$ where each $S_i$ satisfies $|b - b^\prime| \leq An^{-1/(d+1)}$, $|\omega - \omega^\prime| \leq \frac{A|a|}{2R}n^{-1/(d+1)}$ and $\eta = \eta^\prime$ whenever $(\omega,b,\eta),(\omega^\prime,b^\prime,\eta^\prime)\in S_i$.
 
 Add the additional set $S_{2n+1} = S_A^c$ and let $d\tilde\lambda_i = \tilde\lambda(S_i)^{-1}d\tilde\lambda|_{S_i}$ (the conditional probability distribution on the set $A_i$). We write
 \begin{equation}\label{eq_960}
    f(x) = \left[\displaystyle\sum_{i=1}^{2n+1} \tilde\lambda(S_i)\mathbb{E}_{d\tilde\lambda_i}\left(\eta J(\omega,b)\sigma\left(\frac{\omega}{a}\cdot x+b\right)\right)\right].
 \end{equation}
Now we draw samples from each of the sets $S_i$ according to the distributions $\tilde\lambda_i$ to estimate the inner expectations above, which gives us an estimate of $f(x)$.

In particular, we let $c_i = \lceil\tilde\lambda(S_i)n\rceil$ be the number of samples chosen from each set $S_i$. Notice that $\sum_{i=1}^n c_i \leq 3n + 1$ (since the $\tilde\lambda(S_i)$ sum to $1$) and thus the total number of samples is $O(n)$.

So let $(\omega_{ij}, b_{ij}, \eta_{ij})$ with $i = 1,...,2n+1$ and $j = 1,...,c_i$ be random samples from $S_i$ drawn independently from $\tilde\lambda_i$. We form our approximation $\tilde{f}_n\in \Sigma_d^{3n+1}(\sigma)$ as
\begin{equation}\label{eq_968}
 \tilde{f}_n(x) = \displaystyle\sum_{i=1}^{2n+1}\tilde\lambda(S_i) \frac{1}{c_i}\displaystyle\sum_{j=1}^{c_i}\left(\eta_{ij}J(\omega_{ij},b_{ij})D_x^\alpha\sigma\left(\frac{\omega_{ij}}{a}\cdot x+b_{ij}\right)\right).
\end{equation}
Equation \eqref{eq_960} implies that
\begin{equation}
 \mathbb{E}(\tilde{f}_n) = f(x),
\end{equation}
where the expectation is taken over the samples above.

We now proceed by bounding the variance
\begin{equation}
 \mathbb{E}(\|\tilde{f}_n - \mathbb{E}(\tilde{f}_n)\|^2_{H^m(\Omega)}) = \mathbb{E}(\|\tilde{f}_n - f\|^2_{H^m(\Omega)}).
\end{equation}
Equation \eqref{eq_968} along with the independence of the samples $(\omega_{ij},b_{ij},\eta_{ij})$ implies that this variance is given by
\begin{equation}\label{eq_982}
 \mathbb{E}(\|\tilde{f}_n - f\|^2_{H^m(\Omega)}) = \displaystyle\sum_{i=1}^{2n+1} \tilde\lambda(S_i)^2\frac{1}{c_i}V_i,
\end{equation}
where $V_i$ is the $H^m(\Omega)$-variance of the random function
\begin{equation}
 f_{\omega, b, \eta}(x) = \eta J(\omega,b)\sigma\left(\frac{\omega}{a}\cdot x+b\right)
\end{equation}
with $(\omega, b, \eta)$ chosen according to $\tilde\lambda_i$, i.e.
\begin{equation}
 V_i = \mathbb{E}_{d\tilde{\lambda}_i}(\|f_{\omega, b, \eta} - \mathbb{E}_{d\tilde{\lambda}_i}f_{\omega, b, \eta}\|^2_{H^m(\Omega)}).
\end{equation}
We proceed to bound the $V_i$. Consider first $V_{2n+1}$. It is clear that
\begin{equation}
V_{2n+1} \leq \sup_{(\omega, b, \eta)\in S_{2n+1}}\|f_{\omega, b, \eta}\|^2_{H^m(\Omega)} \leq \sup_{\omega, b}\|f_{\omega, b}\|^2_{H^m(\Omega)}.
\end{equation}
where $f_{\omega, b}$ is as in the proof of theorem \ref{approximation_rate_theorem} (just with $\chi(\omega,b)$ removed). The argument culminating in equation \eqref{eq_890} then gives us
\begin{equation}\label{v_bound_1}
 V_{2n+1} \lesssim |\Omega|\|f\|^2_{\mathcal{B}^{m+1}} \leq |\Omega|\|f\|^2_{\mathcal{B}^{m+1+\epsilon}}.
\end{equation}
Here and in the following the implied constant in $\lesssim$ will only depend on $\sigma,m,p$, and $\text{diam}(\Omega)$ (and not $f$ or $n$).

Now consider $V_i$ for $i\leq 2n$. It is clear that
\begin{equation}
 V_i \leq \sup_{(\omega,b,\eta),(\omega^\prime,b^\prime,\eta^\prime)\in S_i} \|f_{\omega,b,\eta} - f_{\omega^\prime,b^\prime,\eta^\prime}\|_{H^m(\Omega)}^2.
\end{equation}
As in the proof of theorem \ref{approximation_rate_theorem}, the Cauchy-Schwartz inequality gives
\begin{equation}
 \|f_{\omega,b,\eta} - f_{\omega^\prime,b^\prime,\eta^\prime}\|_{H^m(\Omega)} \leq |\Omega|^{\frac{1}{2}} \|f_{\omega,b,\eta} - f_{\omega^\prime,b^\prime,\eta^\prime}\|_{W^{m,\infty}(\Omega)},
\end{equation}
and thus it suffices to bound $|D_x^\alpha f_{\omega,b,\eta}(x) - D_x^\alpha f_{\omega^\prime,b^\prime,\eta^\prime}(x)|$ uniformly in $x$ and $ (\omega,b,\eta),(\omega^\prime,b^\prime,\eta^\prime)\in S_i$, for $|\alpha|\leq m$.

So suppose $(\omega,b,\eta)$ and $(\omega^\prime, b^\prime,\eta^\prime)$ are both in $S_i$. We proceed to bound the quantity
\begin{equation}\label{eq_1006}
 |D_x^\alpha f_{\omega,b,\eta}(x) - D_x^\alpha f_{\omega^\prime,b^\prime,\eta^\prime}(x)| = \left|\eta J(\omega,b)D_x^\alpha\sigma\left(\frac{\omega}{a}\cdot x+b\right) - \eta^\prime J(\omega^\prime, b^\prime)D_x^\alpha\sigma\left(\frac{\omega^\prime}{a}\cdot x+b^\prime\right)\right|.
\end{equation}
By the definition of $S_i$, the signs $\eta = \eta^\prime$ are equal. Taking out these signs and taking the derivatives of the $\sigma$ terms gives
\begin{equation}\label{eq_1002}
 |D_x^\alpha f_{\omega,b,\eta}(x) - D_x^\alpha f_{\omega^\prime,b^\prime,\eta^\prime}(x)| \leq |a|^{-|\alpha|}\left|q(\omega,b,x) - q(\omega^\prime,b^\prime,x)\right|
\end{equation}
where (expanding the definition of $J(\omega, b)$ \eqref{definition-of-J})
\begin{equation}
\begin{split}
q(\omega,b,x) &= \omega^{\alpha}J(\omega,b)\sigma^{(|\alpha|)}\left(\frac{\omega}{a}\cdot x+b\right) \\ 
&= (2\pi |\hat{\sigma}(a)|)^{-1}I(p,\Omega,\sigma,f)\omega^\alpha(1 + |\omega|)^{-m} h(b,\omega)^{-1}\sigma^{(|\alpha|)}\left(\frac{\omega}{a}\cdot x+b\right).
\end{split}
\end{equation}
We now differentiate $q$ with respect to $\omega$ and $b$, noting the following bounds.
\begin{equation}
 |(1 + |\omega|)^{-m}\omega^\alpha|,|D_\omega [(1 + |\omega|)^{-m}\omega^\alpha]| \lesssim 1,
\end{equation}
since $|\alpha| \leq m$, and
\begin{equation}
 |h(\omega, b)^{-1}| \leq \left(1 + \max\left(0,|b| - \frac{R|\omega|}{|a|}\right)\right)^{p},
\end{equation}
\begin{equation}
 |D_\omega [h(\omega, b)^{-1}]|,|D_b [h(\omega, b)^{-1}]| \lesssim \left(1 + \max\left(0,|b| - \frac{R|\omega|}{|a|}\right)\right)^{p-1} \leq \left(1 + \max\left(0,|b| - \frac{R|\omega|}{|a|}\right)\right)^{p}
\end{equation}
by the definition of $h$ \eqref{h_definition}. Combining these bounds with the assumption \eqref{growth_condition_two} on the decay of the derivatives of $\sigma$ (note that this is where we need the decay condition on the $m+1$ derivative), we get (using the bound on the normalization constant $I$, \eqref{bound_on_I})
\begin{equation}
 |D_b q(\omega,b,x)| \lesssim C_p\left(1+\left|\frac{\omega}{a}\cdot x+b\right|\right)^{-p}\left(1 + \max\left(0,|b| - \frac{R|\omega|}{|a|}\right)\right)^{p}\|f\|_{\mathcal{B}^{m+1}}.
\end{equation}
As in the proof of theorem \ref{approximation_rate_theorem} we now use the relation in equation \eqref{eq_828} to see that
\begin{equation}
 \left(1+\left|\frac{\omega}{a}\cdot x+b\right|\right)^{-p}\left(1 + \max\left(0,|b| - \frac{R|\omega|}{|a|}\right)\right)^{p} \leq 1,
\end{equation}
and so
\begin{equation}
 |D_b q(\omega,b,x)| \lesssim|f\|_{\mathcal{B}^{m+1}} \leq \|f\|_{\mathcal{B}^{m+1+\epsilon}}.
\end{equation}
Similarly, we obtain, since $|x|\leq R$,
\begin{equation}
 |D_\omega q(\omega,b,x)| \lesssim \frac{R}{|a|}\|f\|_{\mathcal{B}^{m+1}} \leq \frac{R}{|a|}\|f\|_{\mathcal{B}^{m+1+\epsilon}}.
\end{equation}
This means that for fixed $x$, the function $q$ is Lipschitz with respect to both $\omega$ and $b$ with constant proportional to $\|f\|_{\mathcal{B}^{m+1+\epsilon}}$.
Now the diameter bounds on $S_i$ imply that the $\omega$ and $\omega^\prime$, and the $b$ and $b^\prime$ in equation \eqref{eq_1002} differ by at most $\frac{A|a|}{2R}n^{-1/(d+1)}$ and $An^{-1/(d+1)}$, respectively. Combining this with the above Lipschitz bounds, we see that the quantity in equation \eqref{eq_1006} is bounded by
\begin{equation}
 \eqref{eq_1006} \lesssim An^{-1/(d+1)}\|f\|_{\mathcal{B}^{m+1+\epsilon}} 
\end{equation}
uniformly in $x$ and $(\omega,b,\eta), (\omega^\prime,b^\prime,\eta^\prime)\in S_i$,
and thus the variance $V_i$ is bounded by
\begin{equation}\label{v_bound_2}
V_i \lesssim |\Omega|A^2n^{-2/(d+1)}\|f\|_{\mathcal{B}^{m+1+\epsilon}}^2.
\end{equation}
Plugging equations \eqref{v_bound_2} and \eqref{v_bound_1} into equation \eqref{eq_982}, and noting that $c_i \geq n\tilde\lambda(S_i)$, we get
\begin{equation}\label{eq_1021}
 \mathbb{E}(\|\tilde{f}_n - f\|^2_{H^m(\Omega)}) \lesssim |\Omega|\frac{\|f\|_{\mathcal{B}^{m+1+\epsilon}}^2}{n}\left[\tilde\lambda(S_{2n+1}) + \displaystyle\sum_{i=1}^{2n} \tilde\lambda(S_i)A^2n^{-2/(d+1)}\right].
\end{equation}

The final ingradient we need is a bound on the probability measure of $S_{2n+1} = S_A^c$. To do this, we break the set $S_A^c$ into two pieces, $A_1$, where $|\omega|>\frac{A|a|}{2R}$,
and $A_2$, where $|\omega|\leq \frac{A|a|}{2R}$ and $|b| > A$. We get
\begin{equation}\label{eq_1054}
 \tilde\lambda(S_{2n+1}) \leq \frac{1}{I(p,\Omega,\sigma,f)}\left[\int_{A_1}(1 + |\omega|)^m h(b,\omega) |\hat{f}(\omega)| dbd\omega + \int_{A_2}(1 + |\omega|)^m h(b,\omega) |\hat{f}(\omega)| dbd\omega\right].
\end{equation}
By integrating out $b$, we immediately obtain
\begin{equation}
 \int_{A_1}(1 + |\omega|)^m h(b,\omega) |\hat{f}(\omega)| dbd\omega \leq \int_{|\omega| > \frac{A|a|}{2R}} C_p(1 + |\omega|)^{m+1}|\hat{f}(\omega)| d\omega \lesssim \|f\|_{\mathcal{B}^{m+1+\epsilon}}A^{-\epsilon}.
\end{equation}
On the other hand, on $A_2$ we have $|\omega|\leq \frac{A|a|}{2R}$ and $|b| > A$. This implies that
\begin{equation}
 |b| - \frac{R|\omega|}{|a|} \geq \frac{A}{2},
\end{equation}
so that for $|\omega|\leq \frac{A|a|}{2R}$, we have
\begin{equation}
 \int_{|b| > A} h(\omega,b) = \int_{|b| > A}\left(1 + \max\left(0,|b| - \frac{R|\omega|}{|a|}\right)\right)^{-p} \leq \int_{|x| > \frac{A}{2}}\left(1 + x\right)^{-p} \lesssim\left(1 + \frac{A}{2}\right)^{1-p}.
\end{equation}
Integrating in $b$ and then $\omega$, we obtain (recall $A\geq 1$)
\begin{equation}
 \int_{A_2}(1 + |\omega|)^m h(b,\omega) |\hat{f}(\omega)| dbd\omega \lesssim \left(1 + \frac{A}{2}\right)^{1-p} \int_{|\omega| \leq \frac{A|a|}{2R}} (1 + |\omega|)^m|\hat{f}(\omega)| d\omega \lesssim \|f\|_{\mathcal{B}^{m+1+\epsilon}}A^{1-p}.
\end{equation}
Plugging this into equation \eqref{eq_1054} we get (note that equations \eqref{eq_775} and \eqref{bound_on_I} imply that $\|f\|_{\mathcal{B}^{m+1+\epsilon}} \lesssim I(p,\omega,\sigma,f)$)
\begin{equation}
 \tilde\lambda(S_{2n+1}) \lesssim A^{-\min(p-1,\epsilon)}.
\end{equation}
Combining all of this with equation \eqref{eq_1021}, we get
\begin{equation}
 \mathbb{E}(\|\tilde{f}_n - f\|^2_{H^m(\Omega)}) \lesssim |\Omega|\frac{\|f\|_{\mathcal{B}^{m+1+\epsilon}}^2}{n}\left[A^{-\min(p-1,\epsilon)} + A^2n^{-2/(d+1)}\right].
\end{equation}
Optimizing over $A$, we obtain, for $A = n^{2/[(d+1)(2 + \min(p-1,\epsilon))]}$,
\begin{equation}
 \mathbb{E}(\|\tilde{f}_n - f\|^2_{H^m(\Omega)}) \lesssim |\Omega|\|f\|_{\mathcal{B}^{m+1+\epsilon}}^2n^{-1-\frac{2\min(p-1,\epsilon)}{(d+1)(2+\min(p-1,\epsilon))}}.
\end{equation}
This bound on the expectation means that there must exist samples $(\omega_{ij},b_{ij}, \eta_{ij})$ such that $\tilde{f}_n$ defined by \eqref{eq_968} satisfies
\begin{equation}
 \|\tilde{f}_n - f\|^2_{H^m(\Omega)}\lesssim |\Omega|\|f\|_{\mathcal{B}^{m+1+\epsilon}}^2n^{-1-\frac{2\min(p-1,\epsilon)}{(d+1)(2+\min(p-1,\epsilon))}}.
\end{equation}
Thus we finally get
\begin{equation}
   \inf_{f_n\in \Sigma_d^{3n+1}(\sigma)}\|f_n - f\|_{H^m(\Omega)} \leq |\Omega|^{\frac{1}{2}}C(p,m,\text{\normalfont diam}(\Omega),\sigma)\|f\|_{\mathcal{B}^{m+1+\epsilon}}n^{-\frac{1}{2}-\frac{\min(p-1,\epsilon)}{(d+1)(2+\min(p-1,\epsilon))}}
\end{equation}
as desired.

\end{proof}

We can also improve the rate in Theorem \ref{periodic-activation} using a stratified sampling approach. The argument is a straightforward modification of the proof of Theorem \ref{stratified-sampling} beginning with the representation \eqref{periodic_representation}, and we state the result without proof.

\begin{theorem}\label{periodic-improved}
  Let $\Omega\subset \mathbb{R}^d$ be a bounded domain with Lipschitz boundary and $\epsilon > 0$. If the activation function $\sigma\in W^{m,\infty}(\mathbb{R})$ is a non-constant periodic function, we have
 \begin{equation}
  \inf_{f_n\in \Sigma_d^n(\sigma)}\|f - f_n\|_{H^m(\Omega)} \leq |\Omega|^{\frac{1}{2}}C(\sigma)n^{-\frac{1}{2}-\frac{\epsilon}{(d+1)(2+\epsilon)}}\|f\|_{\mathcal{B},m+\epsilon},
 \end{equation}
 for any $f\in \mathcal{B}^{m+\epsilon}$.
\end{theorem}

\section{Conclusion}
We have provided a few new results in the theory of approximation by neural networks. These results improve upon existing results
on the rate of approximation achieved by two-layer neural network as the number of neurons increases
by entending the results to more general activation functions. In particular, we obtain a dimension independent rate of $n^{-\frac{1}{2}}$ for polynomialls decaying activation functions, and a rate of $n^{-\frac{1}{4}}$ for more general bounded activation functions. We also show how a stratified sampling argument can be used to further improve the convergence rate under mild additional assumptions.

There are a few interesting question which remain. On the one hand, we conjecture that Theorem \ref{theorem_no_decay} can be further improved to provide an approximation rate of $n^{-\frac{1}{2}}$. In addition, we would like to determine whether the rate in Theorem \ref{stratified-sampling} can be further improved. In particular, a better rate is obtained for the ReLU activation function in \cite{klusowski2016uniform}. Finally, we propose that a more detailed investigation of the Barrom norm in high dimensions is an important piece in understanding how shallow neural networks defeat the curse of dimensionality.

\section{Acknowledgements}
We would like to thank Prof. Jason Klusowski and Dr. Juncai He for insightful discussions and helpful comments. This work was supported by a Penn State Institute for Cyber Science Seed Grant, the Verne M. Willaman Fund, and the National Science Foundation (Grant No. DMS-1819157).
\bibliographystyle{spmpsci}
\bibliography{references}
\end{document}